\documentclass{amsart}
\usepackage[utf8]{inputenc}
\usepackage[T1]{fontenc}
\usepackage[frenchb]{babel}
\usepackage{lmodern}

\newtheorem{definition}{Définition}[section]
\newtheorem{proposition}[definition]{Proposition}
\newtheorem{lemma}[definition]{Lemme}
\newtheorem{theorem}[definition]{Théorème}
\newtheorem{corollary}[definition]{Corollaire}

\newtheorem*{theorem*}{Théorème}
\newtheorem*{proposition*}{Proposition}

\title[Points rationnels de la fonction Gamma d'Euler]{Points rationnels\\ de la fonction Gamma d'Euler}
\author{Etienne Besson}
\address{Institut Fourier, CNRS UMR 5582\\ Université Grenoble 1\\100 rue des mathématiques, BP 74\\ 38402 Saint-Martin d'Hères\\ France}
\email{etienne.besson@ujf-grenoble.fr}
\urladdr{http://www-fourier.ujf-grenoble.fr/~bessonet/}
\subjclass[2010]{11J72}

\newcommand{\C}{\mathbf{C}}

\newcommand{\Q}{\mathbf{Q}}
\newcommand{\N}{\mathcal{N}}

\renewcommand{\Re}{\operatorname{Re}}
\renewcommand{\Im}{\operatorname{Im}}

\begin{document}

\begin{abstract}
 En s'inspirant des méthodes employées par Masser pour la fonction zêta de Riemann dans son article [\og{}Rational values of the Riemann zeta function\fg{}, \textit{Journal of Number Theory} \textbf{131} (2011), 2037--2046], on démontre un nouveau lemme de zéros pour les polynômes en $z$ et $1/\Gamma(z)$ et on en déduit l'existence pour tout $n\geq2$ d'un réel positif $C(n)$ tel que le nombre de rationnels de dénominateur au plus $D\geq3$ dans l'intervalle $[n-1,n]$ et dont l'image par $\Gamma$ est également un rationnel de dénominateur au plus $D$ soit borné par $C(n)\frac{\log^2(D)}{\log\log(D)}$.
\end{abstract}

\maketitle{}

\section{Introduction}

L'objet de cet article est l'étude des points rationnels du graphe de la fonction Gamma d'Euler, définie comme le prolongement analytique de l'intégrale à paramètre $\Gamma(z)=\int_0^{+\infty}t^{z-1}e^{-t}\,\mathrm{d}t$. Plus précisément, on se place sur un intervalle $[n-1,n]$ de la demi-droite réelle positive et on s'intéresse au nombre $\N(D,n)$ de rationnels $p/q$ de dénominateur $q$ au plus $D$ contenus dans ce segment et tels que $\Gamma(p/q)$ soit également un rationnel de dénominateur au plus $D$.

Pour $D\geq2$, il existe moins de $D^2$ rationnels de dénominateur au plus $D$ dans l'intervalle $[n-1,n]$, donc $\N(D,n)$ est borné par $D^2$, mais on s'attend à un résultat très inférieur: en effet, la fonction Gamma envoie les entiers strictement positifs sur des entiers, mais on ne connaît aucun rationnel (réduit) $p/q$ avec $q\geq2$ tel que $\Gamma(p/q)$ soit rationnel. 

En revanche, il est connu que $\Gamma(\frac{1}{2})=\sqrt{\pi}$, est transcendant, ainsi que $\Gamma(\frac{1}{3})$ et $\Gamma(\frac{1}{4})$ (voir \cite{Chudnovsky}). On peut alors utiliser les équations fonctionnelles de $\Gamma$ pour construire de nombreuses autres valeurs transcendantes, par exemple $\Gamma(\frac{1}{6})$. Par contraste, on ne connaît pas encore la nature de $\Gamma(\frac{1}{5})$ (mais on sait qu'au moins deux nombres parmi $\pi$, $\Gamma(\frac{1}{5})$ et $\Gamma(\frac{2}{5})$ sont algébriquement indépendants, voir \cite[prop. 3.12.1]{Grinspan}).

On propose de démontrer, en suivant la méthode employée par Masser pour la fonction zêta de Riemann dans son article \cite{Masser}, les majorations suivantes:

\begin{theorem*}
 \begin{enumerate}
  \item Il existe une constante absolue $c>0$ telle que, pour tout couple d'entiers $n\geq2$ et $D\geq3$, le nombre $\N(D,n)$ de rationnels $z\in[n-1,n]$ de dénominateur au plus $D$ et tels que $\Gamma(z)$ soit également un rationnel de dénominateur au plus $D$ vérifie:
  \[\N(D,n)\leq c n^4\log^3(n)\frac{\log^2(D)}{\log\log(D)}.\]
  \item  Il existe une constante absolue $c^\prime>0$ telle que, pour tout couple d'entiers $n\geq2$ et $D\geq3$, le nombre $\N^\prime(D,n)$ de rationnels $z\in[n-1,n]$ de dénominateur au plus $D$ et tels que $1/\Gamma(z)$ soit également un rationnel de dénominateur au plus $D$ vérifie:
  \[\N^\prime(D,n)\leq c^\prime n^2\log^3(n)\frac{\log^2(D)}{\log\log(D)}.\]
 \end{enumerate}
\end{theorem*}

Il n'est pas surprenant d'obtenir des résultats de ce type. Par exemple, les résultats de l'article \cite{Surroca} de Surroca, valables pour une classe plus générale de fonctions, conduisent facilement pour la fonction Gamma à la borne $\N(D,n)\leq\delta\cdot(n\log(n)\log(D))^2$, avec une constante $\delta>0$ pouvant être choisie arbitrairement petite, mais seulement pour une sous-suite (non explicite) de valeurs de $D$.

Simultanément à la rédaction de cet article, Boxall et Jones ont également obtenu des résultats plus généraux dans \cite{BoxallJones}. En particulier, ils obtiennent une borne en $C\log^3(D)$ pour le nombre de points rationnels de dénominateur borné par $D$ de la restriction de la fonction Gamma à un compact quelconque (la constante $C$ dépendant du compact).

La stratégie employée ici suit la méthode de Masser et peut être résumée ainsi: dans la partie \ref{sectionlemmezeros}, on utilise les propriétés analytiques de la fonction entière $1/\Gamma$ pour démontrer un lemme de zéros pour les polynômes en $z$ et $1/\Gamma(z)$. Dans la partie \ref{sectiontheoreme}, on utilise la Proposition 2 de \cite{Masser} et le lemme de zéros pour en déduire un résultat plus général sur la répartition des points de degré et de hauteur fixés sur le graphe de $1/\Gamma$. Le théorème annoncé est alors un corollaire de ce résultat.

Je tiens à remercier Tanguy Rivoal, sans l'aide duquel cet article n'aurait pas pu voir le jour.

\section{Propriétés classiques de la fonction Gamma}

Les méthodes employées dans la suite s'appliquent à des fonctions entières avec une infinité de zéros \og{}bien distribués\fg{} en un certain sens, c'est pourquoi on va en général travailler avec $1/\Gamma$ plutôt qu'avec la fonction Gamma elle-même (sachant que déduire les points rationnels de $\Gamma$ de ceux de $1/\Gamma$ ne pose pas de problème).

Cette partie est consacrée à un bref rappel des propriétés de la fonction $1/\Gamma$ (on pourra se référer à \cite{WhittakerWatson} ou \cite{Campbell} pour plus de détails).

\begin{lemma}\label{defgamma}Pour tout $z\in\C$, on a \[\frac{1}{\Gamma(z)}=ze^{\gamma z}\prod_{n\geq1}\left(1+\frac{z}{n}\right)e^{-z/n},\]
 où $\gamma$ est la constante d'Euler-Mascheroni.
\end{lemma}

On déduit de ce produit absolument convergent que $1/\Gamma$ est une fonction entière admettant un zéro d'ordre un en chacun des entiers négatifs ou nuls.

Le lien entre le comportement de $1/\Gamma$ dans la partie droite et dans la partie gauche du plan complexe est donné par la formule des compléments \cite[12.14]{WhittakerWatson}:
\begin{lemma}[Formule des compléments]
 Pour tout $z\in\C$, on a \[\frac{1}{\Gamma(z)\Gamma(1-z)}=\frac{\sin(\pi z)}{\pi}.\]
\end{lemma}
Un développement asymptotique de la fonction pour des valeurs de $z$ éloignées de la demi-droite réelle négative est donné par la formule de Stirling \cite[12.33]{WhittakerWatson}:
\begin{lemma}[Formule de Stirling]
 Pour $z\in\C$, $|\arg(z)|\leq\delta<\pi$ et $|z|\to\infty$, on a:
 \[\frac{1}{\Gamma(z)}=\frac{1}{\sqrt{2\pi}}e^{z}z^{-z+1/2}\left(1+O\left(z^{-1}\right)\right).\]
\end{lemma}

En particulier, la formule est valable dans tout le demi-plan $\Re(z)>0$. En utilisant la formule des compléments, on en déduit une seconde estimation, valable dans le demi-plan $\Re(z)<1$:

\begin{lemma}\label{stirling}
 Pour $z\in\C$, $|z|\to\infty$, on a:
 \[\frac{1}{\Gamma(z)}=\left\{\begin{array}{ll}
    \displaystyle\frac{e^{z}z^{-z+1/2}}{\sqrt{2\pi}}\left(1+O\left(z^{-1}\right)\right) & \text{ si }\Re(z)>0\\
    \displaystyle\sqrt{\frac{2}{\pi}}e^{z-1}(1-z)^{-z+1/2}\sin(\pi z)\left(1+O\left(z^{-1}\right)\right) & \text{ si }\Re(z)<1
                             \end{array}\right.\]
\end{lemma}
                             
De cette estimation, on tire une majoration de la croissance radiale de $1/\Gamma$:

\begin{lemma}\label{croissance}
  Il existe une constante $c$ telle que, pour tout réel positif $R$ assez grand, on ait $\left|1/\Gamma\right|_R\leq R^{cR}$, où $|1/\Gamma|_R$ désigne le maximum du module de $1/\Gamma$ sur le disque fermé centré en $0$ et de rayon $R$.
\end{lemma}

\begin{proof}
 D'après le lemme \ref{stirling}, pour $\Re(z)>0$ et $|z|=R$ assez grand, on a:
 \[\left|\frac{1}{\Gamma(z)}\right|\leq \left|e^{z}z^{-z+1/2}\right|\leq e^R R^{R+1/2} e^{\pi R/2}\leq R^{cR}.\]
 Par ailleurs, pour $\Re(z)<1$ et $|z|=R$ assez grand, on a:
 \[\left|\frac{1}{\Gamma(z)}\right|\leq\left|e^{z-1}(1-z)^{-z+1/2}\sin(\pi z)\right|\leq(1+R)^{R+1/2}e^{\pi R/2}\leq R^{c^\prime R}.\]

\end{proof}

\section{Lemme de zéros pour les polynômes en $z$ et $1/\Gamma(z)$}\label{sectionlemmezeros}

Dans toute la suite, on notera $G(z)=1/\Gamma(z)$ pour alléger les formules. L'objectif de cette partie est la démonstration de l'estimation suivante:
\begin{theorem}\label{lemmedezeros}
  Il existe une constante absolue $c$ telle que, pour tout entier $L\geq1$, pour tout réel $R\geq2$ et pour tout polynôme non nul $P(z,w)$ de degré au plus $L$ suivant chaque variable, la fonction $P(z,G(z))$ admet au plus $cL(L+R)\log(L+R)$ zéros (comptés avec multiplicité) dans le disque $|z|\leq R$.
\end{theorem}

De manière analogue au lemme de \cite{Masser} pour les polynômes en $z$ et $\zeta(z)$, on établit ce résultat par des techniques d'analyse complexe faisant intervenir la distribution des valeurs de la fonction $G$. Ceci diffère fondamentalement des lemmes de zéros obtenus par Gabrielov et Vorobyov dans \cite{GV} et utilisés notamment par Pila dans \cite{Pila3}, qui font intervenir des arguments d'analyse réelle et s'appliquent à des courbes, dites pfaffiennes, satisfaisant certaines équations différentielles algébriques -- alors qu'il est bien connu que la fonction Gamma n'est solution d'aucune telle équation différentielle (voir \cite{Campbell} ou \cite{WhittakerWatson}).

Dans toute la suite de cette partie, pour $X$ et $Y$ des réels supérieurs à $0$, on définit $\mathcal{Z}(X,Y)$ comme étant l'ensemble des $z=x+iy\in\C$ avec $-X\leq x\leq 1$ et $-Y\leq y\leq Y$.

Le lemme suivant nous donne la distribution dans $\mathcal{Z}(X,Y)$ des solutions des équations $G(z)=w$ pour les petites valeurs de $w$.
\begin{lemma}\label{lemmetech4}
 Il existe une constante absolue $r_0>0$ telle que, pour tout $w\in\C$ avec $|w|\leq r_0$ et pour $X\geq0$ tendant vers l'infini, le nombre $\N(X,w)$ de solutions (comptées avec multiplicité) à l'équation $G(z)=w$ avec $z\in\mathcal{Z}(X,1)$ vérifie
$\N(X,w)=X+O(1)$,
et cela uniformément en $w$.
\end{lemma}

Ce lemme est l'analogue pour $G$ de la proposition (22) de l'article \cite{BLL} de Bohr, Landau et Littlewood sur la distribution des solutions de $\zeta(s)=a$. Contrairement au cas de la fonction zêta, l'existence de la formule de Stirling nous permet d'estimer avec précision le comportement de la fonction, ce qui donne lieu à une démonstration beaucoup plus directe.

\begin{proof}
 Soit $w\in\C$. D'après le théorème des résidus, le nombre de solutions à l'équation $G(z)=w$ avec $z\in\mathcal{Z}(X,1)$ est donné par:
 \begin{equation}\label{residupourtech4}\N(X,w)=\frac{1}{2\pi}\Im\int_\gamma\frac{G^\prime(z)}{G(z)-w}\mathrm{d}z,\end{equation}
 où $\gamma$ est le contour de $\mathcal{Z}(X,1)$, orienté positivement et contournant les éventuels zéros de $G(z)-w$ par de petits arcs de cercle passant à l'extérieur de $\mathcal{Z}(X,1)$.

Remarquons que, d'après la proposition \ref{defgamma}, les solutions de l'équation $G(z)=0$ sont tous les entiers négatifs ou nuls, si bien qu'on a, pour $w=0$, l'estimation :
\begin{equation}\label{NX0}\N(X,0)=\frac{1}{2\pi}\Im\int_\gamma\frac{G^\prime(z)}{G(z)}\mathrm{d}z=X+O(1).\end{equation}

On va montrer que, pour peu qu'on choisisse $w$ de module inférieur à un $r_0$ qu'on déterminera par la suite, il n'y a jamais de zéros de $G(z)-w$ sur le segment vertical $\Re(z)=1$, ni sur les segments horizontaux $\Im(z)=\pm 1$. On va ensuite montrer que, quitte à ajuster légèrement la valeur de $X$, il n'y en a pas non plus sur le segment vertical $\Re(z)=-X$. On obtiendra alors le résultat souhaité en comparant les intégrales \eqref{residupourtech4} et \eqref{NX0}.

 \subsection*{Estimation de $G$ sur le segment $\Re(z)=1$.}

 Soit $z\in[1-i,1+i]$. On a
 \[G(z)=ze^{\gamma z}\prod_{n\geq1}\left(1+\frac{z}{n}\right)e^{-\frac{z}{n}},\]
 donc
 \[|G(z)|  = |z|e^{\gamma x}\prod_{n\geq1}\left|1+\frac{z}{n}\right|e^{-\frac{x}{n}}
           \geq |x|e^{\gamma x}\prod_{n\geq1}\left|1+\frac{x}{n}\right|e^{-\frac{x}{n}},\]
 où $x=\Re(z)=1$. On a:
 \begin{align*}
  \log\prod_{n\geq1}\left|1+\frac{x}{n}\right|e^{-\frac{x}{n}} 
         & = \sum_{n\geq1}\left(\log\left|1+\frac{x}{n}\right|-\frac{x}{n}\right)
          = \sum_{n\geq1}\left(\log\left(1+\frac{x}{n}\right)-\frac{x}{n}\right)\\
         & \geq -\frac{1}{2}\sum_{n\geq1}\left(\frac{x}{n}\right)^2
         \geq -\frac{x^2}{2}\zeta(2)
          > -1,
 \end{align*}
 d'où
 $|G(z)| \geq |x|e^{\gamma x}e^{-3} > \frac{1}{2}$.

 Ainsi, si on prend $|w|\leq r_0\leq\frac{1}{2}$, il n'y aura pas de zéro de $G(z)-w$ à contourner sur le segment $\Re(z)=1$.

 \subsection*{Estimation de $G$ sur les segments horizontaux $\Im(z)=\pm 1$.}

 La formule de Stirling pour le demi-plan $\Re(z)<1$ (lemme \ref{stirling}) nous permet d'obtenir des minorations valables pour $\Re(z)$ suffisamment éloignée de $0$. Resteront à considérer les points à distance bornée de l'origine, que l'on pourra traiter par un calcul direct.
 
 D'après le lemme \ref{stirling}, on a, pour $y=\Im(z)=\pm1$ et $x=\Re(z)<1$ tendant vers $-\infty$,
$G(z)=\sqrt{\frac{\pi}{2}}e^{z-1}(1-z)^{-z+1/2}\sin(\pi z)\left(1+O\left(z^{-1}\right)\right)$.

 En particulier, il existe $R\geq1$ tel que, pour $x\leq-R$, on ait:
\begin{align*}
 \left|G(z)\right| & \geq\frac{1}{2}\left|\sqrt{\frac{2}{\pi}}e^{z-1}(1-z)^{-z+1/2}\sin(\pi z)\right|\\
                   & \geq(2\pi)^{-1/2}e^{x-1}(1-x)^{-x+1/2}e^{y\arg(1-z)}\sinh|\pi y|\\
                   & \geq(2\pi e)^{-1/2}\left(\frac{1-x}{e}\right)^{-x+1/2}e^{y\arg(1-z)}\sinh|\pi y|\\ 
                   & \geq(2\pi e)^{-1/2}\sinh(\pi) e^{-\pi/2}\left(\frac{1+R}{e}\right)^{R+1/2}.
\end{align*}
Par ailleurs, pour $-R\leq x<1$, on revient au lemme \ref{defgamma}:
$G(z) = \allowbreak ze^{\gamma z}\prod_{n\geq1}\allowbreak \left(1+\frac{z}{n}\right)e^{-\frac{z}{n}}$,
donc

\begin{equation}\label{modulegamma}\left|G(z)\right| = |z|e^{\gamma x}\prod_{n\geq1}\left|1+\frac{z}{n}\right|e^{-\frac{x}{n}}\end{equation}

Ici, $x\in[-R,1[$ et $y$ est fixé égal à $\pm 1$ donc $z$ est borné. Soit $n_0$ un entier fixé supérieur à cette borne; alors, pour $n\geq n_0$, on a 
$\left|1+\frac{z}{n}\right| \geq \left|1+\frac{x}{n}\right| = 1+\frac{x}{n} \geq 0$,
d'où
\[\prod_{n\geq n_0}\left|1+\frac{z}{n}\right|e^{-\frac{x}{n}} \geq \prod_{n\geq n_0}\left(1+\frac{x}{n}\right)e^{-\frac{x}{n}}
                                                              \geq e^{-x^2/2\,\zeta(2)}
                                                              \geq e^{-R^2/2\,\zeta(2)}.\]

Et, pour $n<n_0$, on a $\left|1+\frac{z}{n}\right|\geq  \frac{|y|}{n}=\frac{1}{n}$, donc en réinjectant dans \eqref{modulegamma}, il vient
\[\left|G(z)\right| \geq e^{-\gamma R-R^2/2\,\zeta(2)}\prod_{1\leq n<n_0}\frac{1}{n}\,e^{-\frac{1}{n}}.\]

Au final, pour $y=\pm1$ et $-X\leq x< 1$, on a:

\[\left|G(z)\right| \geq
 \left\{\begin{array}{ll}
    \displaystyle (2\pi e)^{-1/2}\sinh(\pi) e^{-\pi/2}\left(\frac{1+R}{e}\right)^{R+1/2} & \text{si }-X\leq x\leq -R\\
    \displaystyle e^{-\gamma R-R^2/2\,\zeta(2)}\prod_{1\leq n<n_0}\frac{1}{n}\,e^{-\frac{1}{n}} & \text{si }-R\leq x<1.
 \end{array}\right.\]

Toutes ces estimations sont indépendantes de $X$, donc on peut choisir une fois pour toutes $r_0$ vérifiant simultanément les trois conditions suivantes:

\[r_0<\left\{\begin{array}{l}\displaystyle\frac{1}{2}\\
                                \displaystyle(2\pi e)^{-1/2}\sinh(\pi) e^{-\pi/2}\left(\frac{1+R}{e}\right)^{R+1/2}\\
                                \displaystyle e^{-\gamma R-R^2/2\,\zeta(2)}\prod_{1\leq n<n_0}\frac{1}{n}\,e^{-\frac{1}{n}}\end{array}\right.\]

Dans ce cas, pour tout $w$ de module inférieur à $r_0$, on est certain que l'équation $G(z)=w$ n'aura pas de solutions sur les trois côtés $\Re(z)=1$ et $\Re(z)=\pm1$ du rectangle $\mathcal{Z}(X,1)$, ce qui nous garantit que le chemin $\gamma$ suit ces trois segments sans contourner de points.

Reste maintenant à traiter le cas du dernier segment du bord de $\mathcal{Z}(X,1)$, celui donné par $x=-X$.

 \subsection*{Cas du segment vertical $\Re(z)=-X$.} D'après l'estimation de Stirling sur le demi-plan $\Re(z)<1$ (lemme \ref{stirling}), il existe $R$ tel que, pour $X\geq R$ on ait:

 \begin{align*}
  |G(-X+iy)| & \geq(2\pi)^{-1/2}e^{-X-1}(1+X)^{X+1/2}e^{y\arg(1+X-iy)}|\sin(\pi z)| \\
                       & \geq(2\pi)^{-1/2}\left(\frac{1+R}{e}\right)^{R+1/2}e^{-1/2-\pi/2}|\sin(\pi X)|,
 \end{align*}
 donc si $X$ n'est pas trop proche d'un nombre entier (disons, $|X-n|\geq\frac{1}{4}$ pour tout entier $n$), on a:
 \begin{equation}\label{minorationverticale}|G(-X+iy)|\geq (4\pi e)^{-1/2}e^{-\pi/2}\left(\frac{1+R}{e}\right)^{R+1/2}.\end{equation}
 Prenons donc $r_0$ vérifiant, en plus des conditions précédentes, la borne \eqref{minorationverticale}.

 \subsection*{Fin de la démonstration.} Pour $r_0$ vérifiant les conditions ci-dessus, $X$ assez grand tel que $|X-n|\geq\frac{1}{4}$ pour tout $n$ et $|w|\leq r_0$, d'après les formules \eqref{residupourtech4} et \eqref{NX0}, on a: 
 \begin{align*}
  \N(X,0)-\N(X,w) & = \frac{1}{2\pi}\Im\int_\gamma\frac{G^\prime(z)}{G(z)}-\frac{G^\prime(z)}{G(z)-w}\mathrm{d}z\\
             & = \frac{1}{2\pi}\Im\int_\gamma\frac{\mathrm{d}}{\mathrm{d}z}\left(\log\left(1+\frac{w}{G(z)}\right)\right)\mathrm{d}z
               = 0,
 \end{align*}
 et comme on a $\N(X,0)=X+O(1)$, on en déduit que $\N(X,w)=X+O(1)$.

 Pour généraliser le résultat à n'importe quel $X$ assez grand, on encadre $X$ par $X^-$ et $X^+$ tels que $|X^{\pm}-X|\leq\frac{1}{2}$ et que $|X^\pm-n|\geq\frac{1}{4}$ pour tout entier $n$. D'après le calcul ci-dessus, on a $\N(X^\pm,w)=X^\pm+O(1)=X+O(1)$, et $\N(X^-,w)\leq \N(X,w)\leq \N(X^+,w)$ donc, par encadrement, $\N(X,w)=X+O(1)$.\end{proof}

L'objet du lemme suivant est la construction d'un ensemble de points d'interpolation ayant de \og{}bonnes\fg{} propriétés d'espacement, c'est-à-dire tels qu'il sera facile d'estimer les facteurs apparaissant dans la formule d'interpolation de Lagrange.

\begin{lemma}\label{wpoints}
  Il existe des constantes $r_0$ et $c_0$ telles que, pour tout entier $L\geq2$, il existe des nombres complexes $w_l\ (l=0,\dots,L)$ et $z_{k,l}\ (k,l=0,\dots,L)$ vérifiant, pour tous $k,l$, $G(z_{k,l})=w_l$,
  avec:
  \begin{enumerate}
    \noindent\begin{minipage}[b]{0.5\linewidth}
    \item \label{wpetits}$|w_l|\leq r_0$ pour tout $l$;
    \end{minipage}
    \begin{minipage}[b]{0.45\linewidth}
    \item \label{wloin}$|w_l-w_i|\geq cL^{-1/2}$ pour tous $l\neq i$;
    \end{minipage}\hfill

    \noindent\begin{minipage}[b]{0.5\linewidth}
    \item \label{zloinde1}$|z_{k,l}-1|\geq 1$ pour tous $k,l$;
    \end{minipage}
    \begin{minipage}[b]{0.45\textwidth}
    \item \label{zpetits}$|z_{k,l}|\leq c_0L$ pour tous $k,l$;
    \end{minipage}\hfill
    \item \label{zloin}$\prod_{j\neq k}|z_{k,l}-z_{j,l}|\geq (L-1)!$ pour tous $k,l$.
  \end{enumerate}
\end{lemma}

\begin{proof}
 On considère le $r_0$ donné par le lemme \ref{lemmetech4}. Pour tout $|w|\leq r_0$, on a donc une estimation du nombre de solutions $z$ à l'équation $G(z)=w$ dans la région $\mathcal{Z}(X,1)$ (on va choisir $X$ convenable par la suite). Quitte à éliminer un nombre dénombrable de $w$, on peut supposer que ces solutions sont toutes de multiplicité $1$. En effet, s'il existe $z$ de multiplicité $\geq2$, alors on a $G^\prime(z)=0$, ce qui ne peut arriver que pour un nombre au plus dénombrable de $z$. Il suffit donc d'éviter les $w=G(z)$ correspondants, qui sont en nombre au plus dénombrable.
  
  Considérons donc un $w$ hors de cet ensemble dénombrable et de module inférieur à $r_0$. D'après le lemme \ref{lemmetech4}, il existe $c>0$ tel que, pour $X\geq cL$, on a au moins $L+1$ solutions $z_0,\dots,z_L$ (qui sont donc deux à deux distinctes) à l'équation $G(z)=w$ dans la bande $\mathcal{Z}(X,1)$. Ces solutions vérifient alors la condition \ref{zpetits} pour une certaine constante $c_0$. De plus, les $z_i$ sont pas trop proches les uns des autres. En effet, quitte à prendre pour $X$ un multiple suffisamment grand de $L$ (ce qui se traduit par une augmentation de $c_0$ dans \ref{zpetits}), on peut choisir les $z_i$ dans des disques de rayon $\frac{1}{4}$ centrés autour d'entiers négatifs deux à deux distincts. On va donc avoir, pour tout $k$:
  \[\prod_{j\neq k}|z_j-z_k| \geq \frac{1}{2}\times\left(1+\frac{1}{2}\right)\times\left(2+\frac{1}{2}\right)\times\cdots\times\left(L-1+\frac{1}{2}\right)\\
                             \geq (L-1)!,
  \]
ce qui démontre la condition \ref{zloin}. La condition \ref{zloinde1} est immédiate si $X$ est un assez grand multiple de $L$ (ce \og{}assez grand\fg{} étant uniforme en $w$).
  
  Pour démontrer le lemme, il suffit donc de choisir arbitrairement des $w_l$ dans le disque défini par la condition \ref{wpetits}, vérifiant la propriété de séparation \ref{wloin} et évitant un nombre au plus dénombrable de points pathologiques, ce qui ne pose aucune difficulté en choisissant $c$ suffisamment petit par rapport à $r_0$.\end{proof}

Passons à la démonstration du lemme de zéros proprement dit (le théorème \ref{lemmedezeros}).

\begin{proof}
  Quitte à multiplier $P$ par une constante, on peut supposer que le maximum des modules de ses coefficients est $1$. Posons $F(z) = P(z,G(z))$ et notons $z_1,\dots,z_N$ les zéros de $F$ comptés avec multiplicité. Considérons alors la fonction entière
  $\phi(z)=F(z)\prod_{n=1}^N(z-z_n)^{-1}$.
  
 Notons ${S}=R+c_0L$ pour le $c_0$ du lemme \ref{wpoints}. D'après le principe du maximum, on a:
 \begin{equation}\label{phirleqphi5r}|\phi|_{{S}}\leq|\phi|_{5{S}}.\end{equation}
 
 D'une part, d'après la proposition \ref{croissance}, $|G|_{5{S}}\leq (5{S})^{5c{S}}\leq{S}^{c^\prime{S}}$, d'où $|F|_{5{S}}\leq {S}^{c{S}L}$, puis:
 \begin{equation}\label{phi5r}|\phi|_{5{S}}\leq{S}^{c{S}L}(4{S})^{-N}.\end{equation}
 
 D'autre part, pour tout $z$ vérifiant $|z-1|\geq1$ et $|z|\leq{S}$, on aura
 $|F(z)| = |\phi(z)|\prod_{n=1}^N|z-z_n| \leq |\phi|_{{S}}(2{S})^N$.
 On déduit alors es inégalités \eqref{phirleqphi5r} et \eqref{phi5r} que $|F(z)|\leq{S}^{c{S}L}2^{-N}$, et ceci s'appliquant en particulier au cas des $z=z_{k,l}$ du lemme \ref{wpoints}, on a donc, pour tous $k,l=0,\dots,L$:
 \begin{equation}\label{majoreP}|P(z_{k,l},w_l)|\leq{S}^{c{S}L}2^{-N}.\end{equation}
  
  On va maintenant utiliser deux fois la formule d'interpolation de Lagrange pour les points du lemme \ref{wpoints}. D'abord:
  \[P(z,w)=\sum_{l=0}^L\left(\prod_{0\leq i\leq L, i\neq l}\frac{w-w_i}{w_l-w_i}\right)P(z,w_l),\]
  puis
  \[P(z,w_l)=\sum_{k=0}^L\left(\prod_{0\leq j\leq L, j\neq k}\frac{z-z_{j,l}}{z_{k,l}-z_{j,l}}\right)P(z_{k,l},w_l).\]
  Au final:
  \[P(z,w)=\sum_{k,l=0}^L\left(\prod_{0\leq i\leq L, i\neq l}\frac{w-w_i}{w_l-w_i}\cdot\prod_{0\leq j\leq L, j\neq k}\frac{z-z_{j,l}}{z_{k,l}-z_{j,l}}\right)P(z_{k,l},w_l).\]
  
  Le coefficient d'indices $(p,q)$ de $P$ est donc donné par:
  \begin{align*}a_{p,q}= \sum_{k,l=0}^L& \left((-1)^{p+q}\, \sigma_{L-p}(w_0,\dots,\widehat{w_l},\dots,w_L)\,\sigma_{L-q}(z_{0,l},\dots,\widehat{z_{k,l}},\dots,z_{L,l})\hspace{-3ex}\phantom{\prod_a}\right.\\
  & \left.\prod_{0\leq i\leq L, i\neq l}(w_l-w_i)^{-1}\prod_{0\leq j\leq L, j\neq k}(z_{k,l}-z_{j,l})^{-1}\right)P(z_{k,l},w_l),\end{align*}
  où $(x_0,\dots,\widehat{x_k},\dots,x_L)$ désigne le $L$-uplet obtenu en omettant l'entrée $x_k$ dans $(x_0,\dots,x_L)$ et où les $\sigma_p$ sont les fonctions symétriques élémentaires
  $\sigma_p(x_1,\dots,x_n)=\sum_{1\leq i_1<\cdots<i_p\leq n}x_{i_1}\cdots x_{i_p}$.
  
    On utilise maintenant les conditions du lemme \ref{wpoints} sur $w_l$ et $z_{k,l}$ ainsi que l'inégalité \eqref{majoreP} pour borner cette expression:
  \begin{align*}
    |a_{p,q}| & \leq (L+1)^2 \binom{L}{L-p}\binom{L}{L-q}\max_l|w_l|^{L-p}\max_{k,l}|z_{k,l}|^{L-q}\\
               & \qquad \left(\min_{i\neq l} |w_l-w_i|\right)^{-L} \min_{k,l}\left(\prod_{0\leq j\leq L, j\neq k}|z_{k,l}-z_{j,l}|\right)^{-1}|P(z_{k,l},w_l)|\\
              & \leq \frac{(L+1)^2\,2^{2L} r_0^{L-p}(c_0L)^{L-q}}{(cL^{-1/2})^{L}(L-1)!}{S}^{c{S}L}2^{-N}
               \leq L^{c_1L}{S}^{c{S}L}2^{-N}\\
              & \leq {S}^{c_2{S}L}2^{-N},
  \end{align*}
  car ${S}=R+c_0L$.
  
  Comme $P$ admet un coefficient de module $1$, on en déduit que $2^N\leq{S}^{c_2{S}L}$, d'où:
  \[N \leq c_3{S}L\log(S)\leq c_3(R+c_0L)L\log(R+c_0L)\leq cL(R+L)\log(R+L).\qedhere\]
\end{proof}

Cette borne ne semble pas tout à fait optimale; on s'attendrait plutôt à une borne en $cL(R+L)$, qui serait alors optimale relativement à chacun des paramètres $L$ et $R$:
\begin{itemize}
 \item pour $L\geq2$, si $P(z,w)$ est un polynôme de degré au plus $L$ en chaque variable, dire que $P(z,G(z))$ s'annule à l'ordre $k$ en $z=0$ revient à écrire un système linéaire homogène de $k$ équations à $(L+1)^2$ inconnues (les coefficients de $P$). Pour $k\leq(L+1)^2-1$, ce système possède des solutions non nulles, c'est-à-dire qu'il existe $P$ non nul de degré au plus $L$ en chaque variable tel que $P(z,G(z))$ ait un zéro d'ordre au moins $L^2+2L$ en $z=0$;
 \item pour $L$ fixé, on peut considérer le polynôme $P(z,w)=w^L$. Pour tout $R\geq3$, $P(z,G(z))$ possède approximativement $R$ zéros, chacun de multiplicité $L$, dans le disque de rayon $R$.
\end{itemize}

Le facteur $\log(R+L)$ supplémentaire reflète le relatif \og{}manque\fg{} de zéros de la fonction $G$. Dans \cite{Masser}, Masser arrivait à obtenir une borne optimale pour les polynômes en $z$ et $\zeta(z)$ en utilisant le fait qu'on peut trouver $R$ zéros de la fonction $\zeta$ dans un disque de rayon $c\frac{R}{\log(R)}$. Pour obtenir le même nombre de zéros pour la fonction $G$, il est nécessaire d'agrandir ce rayon jusqu'à $cR$, et c'est ce qui fait apparaître le facteur en trop. Il ne semble pas facile de contourner ce problème, mais on va voir dans la suite que son impact sur le résultat final n'est pas significatif.

Enfin, même si on ne s'en servira pas ici, notons que l'on peut déduire de la proposition une estimation similaire pour les zéros de $P(z,\Gamma(z))$:

\begin{corollary}
Il existe une constante absolue $c$ telle que, pour tout entier $L\geq1$, pour tout réel $R\geq2$ et pour tout polynôme non nul $P(z,w)$ de degré au plus $L$ suivant chaque variable, la fonction $P(z,\Gamma(z))$ admet au plus $cL(L+R)\log(L+R)$ zéros (comptés avec multiplicité) dans le disque $|z|\leq R$.
\end{corollary}

\begin{proof}
 On peut écrire $P(z,\Gamma(z))=\Gamma(z)^LQ(z,G(z))$ avec $Q$ un polynôme de degré au plus $L$ en chaque variable. Il suffit alors d'appliquer le lemme de zéros au polynôme $Q$ (le facteur $\Gamma(z)^L$ ne faisant pas apparaître de zéros supplémentaires).
\end{proof}

\section{Points rationnels de la fonction Gamma}\label{sectiontheoreme}

Dans toute la suite, si $S$ est une partie finie de $\C^2$, on note $\omega(S)$ le minimum des degrés des courbes algébriques qui passent par tous les points de $S$. Si $\alpha$ est un nombre algébrique de degré $d$, de polynôme minimal $P(x)=a_dx^d+\cdots+a_0$, on pose
$H(\alpha)=\left(a_d\prod_{P(\alpha^\prime)=0}\max(1,|\alpha^\prime|)\right)^{1/d}$
la \emph{hauteur absolue non-logarithmique} de $\alpha$ (voir \cite{heights}). En particulier, si $\alpha=\frac{p}{q}$ est rationnel, alors $H(\alpha)=\max(|p|,|q|)$.

L'objectif de cette partie est de démontrer le théorème suivant, qui estime la répartition des points algébriques le long du graphe de la fonction $G$ en fonction de leur degré et de leur hauteur:

\begin{theorem}\label{maintheorem}
  Soient $d\geq1$  et $n\geq2$ des entiers et $H\geq 3$ un réel. Alors le nombre $\N(n,d,H)$ de nombres complexes $z$ vérifiant:
  \begin{enumerate}
  \noindent\begin{minipage}{0.45\textwidth}
    \item $|z-(n-\frac{1}{2})|\leq\frac{1}{2}$;
  \end{minipage}\hfill
  \begin{minipage}{0.45\textwidth}
    \item $[\Q(z,G(z)):\Q]\leq d$;
  \end{minipage}
  \item $H(z)\leq H$, $H(G(z))\leq H$,
  \end{enumerate}
  est majoré par $cn^2\log(n)\frac{\left(d^2\log(H)\right)^2}{\log(d\log(H))}$, où $c$ est une constante positive absolue.
\end{theorem}

La démonstration se déroule de la façon suivante. Soit $S\subset\C^2$ l'ensemble des couples $(z,G(z))$ où $z$ satisfait les conditions 1, 2 et 3. On va commencer par estimer $\omega(S)$, puis utiliser notre lemme de zéros (théorème \ref{lemmedezeros}) pour en déduire une borne sur le cardinal de $\mathcal{S}$.

Pour l'estimation de $\omega(S)$, on va utiliser le résultat suivant, qui est une reformulation par Masser d'un résultat dû à Bombieri et Pila \cite{bombieripila}:

\begin{proposition}[{\cite[Proposition 2]{Masser}}]\label{bombieripila}
  Soient $d\geq1$ et $T\geq\sqrt{8d}$ des entiers, $A>0$, $Z>0$, $M>0$ et $H\geq 1$ des réels, $f_1$ et $f_2$ des fonctions analytiques sur un voisinage ouvert du disque $|z|\leq 2Z$ et bornées par $M$ sur ce disque. Soit $\mathcal{Z}$ une partie finie de $\C$ vérifiant les propriétés suivantes:
  \begin{enumerate}
    \item $|z|\leq Z$ pour tout $z\in\mathcal{Z}$;
    \item $|z-z^\prime|\leq A^{-1}$ pour tous $z,z^\prime\in\mathcal{Z}$;
    \item $[\Q(f_1(z),f_2(z)):\Q]\leq d$ pour tout $z\in\mathcal{Z}$;
    \item $H(f_1(z))\leq H$, $H(f_2(z))\leq H$ pour tout $z\in\mathcal{Z}$.
  \end{enumerate}
  On suppose enfin que la condition suivante est vérifiée:
  \[(AZ)^T>(4T)^{96d^2/T}(M+1)^{16d}H^{48d^2}.\]
  Alors, en notant $\mathbf{f}=(f_1,f_2)$, on a $\omega(\mathbf{f}(\mathcal{Z}))\leq T$.
\end{proposition}

\begin{proof}[Démonstration du théorème]
  Soit $T\geq\sqrt{8d}$ un entier dont on précisera le choix par la suite. On va appliquer la proposition \ref{bombieripila} à l'ensemble $\mathcal{Z}$ des points $z\in\C$ satisfaisant les hypothèses du théorème et aux fonctions entières $f_1(z)=z$ et $f_2(z)=G(z)$. On va avoir $\omega(\mathbf{f}(\mathcal{Z}))\leq T$, sous réserve que les conditions suivantes soient vérifiées:
  \begin{enumerate}
    \item $\mathcal{Z}$ est contenu dans un disque de diamètre $A^{-1}$;
    \item $[\Q(z,G(z)):\Q]\leq d$ pour tout $z\in \mathcal{Z}$;
    \item $H(z)\leq H$, $H(G(z))\leq H$ pour tout $z\in \mathcal{Z}$;
    \item $(AZ)^T>(4T)^{96d^2/T}(M+1)^{16d}H^{48d^2}$,
  \end{enumerate}
  où $Z\geq n$ et $M$ est un majorant de $|z|$ et de $|G(z)|$ sur le disque $|z|\leq2Z$.
  
  D'après les hypothèses du théorème, on a directement la condition 1 avec $A=1$, la condition 2 ainsi que la condition 3.
    D'après la proposition \ref{croissance}, on peut prendre $M\leq Z^{cZ}$ pour une certaine constante $c$, donc:
  \[    (4T)^{96d^2/T}(M+1)^{16d}H^{48d^2} \leq (4T)^{96d^2/T}Z^{c^\prime d Z}H^{48d^2}
                                       \leq c^{d^2}Z^{c^\prime d Z}H^{48d^2}\]
  et, pour que la condition 4 soit réalisée, il suffit donc d'avoir
  $cd^2+(cdZ-T)\log(Z) + 48d^2\log(H)\leq 0$.
  Si on prend $T\geq2cdZ$, il suffit désormais d'avoir $d(c+48\log(H))\leq cZ\log(Z)$, ce qui se ramène à:
  \begin{equation}\label{choixdeZ}cd\log(H)\leq Z\log(Z),\end{equation}
  ce qui peut être obtenu en choisissant $Z$ supérieur à un multiple constant assez grand de $\frac{d\log(H)}{\log(d\log(H))}$. En effet, si $Z\geq\lambda\frac{d\log(H)}{\log(d\log(H))}$, alors:
  \begin{align*}
    Z\log(Z) & \geq \lambda\frac{d\log(H)}{\log(d\log(H))}\left(\log(d\log(H)) + \log(\lambda)- \log\log(d\log(H))\right)\\
             & \geq \lambda d\log(H)\left(1+\frac{\log(\lambda)}{\log(d\log(H))} - \frac{\log\log(d\log(H))}{\log(d\log (H))}\right)\\
             & \geq \lambda d\log(H)\left(1- \frac{\log\log(d\log(H))}{\log(d\log(H))}\right)
             \geq \lambda d\log(H)\left(1-\frac{1}{e}\right).
  \end{align*}
  
  On pose donc $Z=\lambda\frac{d\log(H)}{\log(d\log(H))}+n$ pour une valeur de $\lambda$ suffisamment grande, de façon à assurer simultanément \eqref{choixdeZ} et la condition $Z\geq n$. On prend alors pour $T$ un entier supérieur ou égal à $2cdZ$, mais du même ordre de grandeur (par exemple, la partie entière supérieure de $2cdZ$). La condition $T\geq\sqrt{8d}$ sera alors vérifiée, quitte à augmenter encore la valeur de $\lambda$.
  
  On peut alors appliquer la proposition \ref{bombieripila}, qui nous dit que $\omega(\mathbf{f}(\mathcal{Z}))\leq T$, c'est-à-dire qu'il existe un polynôme $P$ en deux variables, de degré total au plus $T$, tel que tous les points de $\mathbf{f}(\mathcal{Z})$ soient des zéros de $P$. Autrement dit, pour tout $z$ vérifiant les hypothèses du théorème, on a $P(z,G(z))=0$.
  
  Le degré de $P$ en chaque variable est majoré par $T$, donc on peut utiliser le théorème \ref{lemmedezeros} pour borner le nombre de zéros de $P(z,G(z))$ dans un disque de rayon $n$, donc a fortiori le nombre de points de l'ensemble $\mathcal{Z}$, c'est-à-dire les points satisfaisant les hypothèses du théorème:
  \begin{align*}
    \N(n,d,H) & \leq cT(T+n)\log(T+n)\\
             & \leq c\left(\frac{d^2\log(H)}{\log(d\log(H))}+n\right)^2\log\left(\frac{d^2\log(H)}{\log(d\log(H))}+n\right)\\
             & \leq c (n^2\log(n))\frac{\left(d^2\log(H)\right)^2}{\log(d\log(H))}.\qedhere
  \end{align*}  
\end{proof}

Si l'on parvenait à supprimer le facteur $\log(L+R)$ dans la proposition \ref{lemmedezeros}, cela se traduirait ici par un dénominateur en $(\log(d\log(H)))^2$ au lieu de $\log(d\log(H))$.

On peut maintenant démontrer les résultats annoncés comme corollaires du théorème \ref{maintheorem}.

\begin{proof}
 \begin{enumerate}
  \item Si $z$ est un rationnel compris entre $n-1$ et $n$ et de dénominateur au plus $D$, le numérateur de $z$ est au plus $nD$ donc on a $H(z)\leq nD$. Par ailleurs, $0\leq \Gamma(z)\leq\Gamma(n)=(n-1)!$ donc $H(G(z))\leq (n-1)!D$.
  
  On applique donc le théorème \ref{maintheorem} avec $d=1$ et $H=(n-1)!D$:
  \begin{align*}
    \N(D,n) & \leq c (n^2\log(n))\frac{\log^2 ((n-1)!D)}{\log\log ((n-1)!D)}\\
           & \leq c (n^2\log(n))\frac{(n\log(n))^2\log^2 (D)}{\log\log (D)}
            \leq c (n^4\log^3(n))\frac{\log^2(D)}{\log\log(D)}.
  \end{align*}
  
  \item On fait le même calcul, mais $G$ étant bornée par $2$ sur $[1,+\infty[$, la borne pour le numérateur de $G(z)$ est cette fois $2D\leq nD$, donc on applique le théorème avec $H=nD$, ce qui donne une borne en $c^\prime n^2\log^3 (n)\frac{\log^2(D)}{\log\log(D)}$. \end{enumerate}
\end{proof}

\nocite{*}

\bibliographystyle{smfalpha}
\bibliography{biblio}

\end{document}